\documentclass[11pt]{article}
\usepackage{fullpage}
\usepackage{times}
\usepackage{amsmath,amssymb,amsthm,url}
\usepackage[utf8]{inputenc}
\usepackage[english]{babel}
\usepackage{comment}
\usepackage{bbm}
\usepackage{enumerate}
\usepackage{bm}
\usepackage{graphicx}
\usepackage{mathrsfs}
\usepackage{float}
\usepackage{authblk}

\usepackage[colorlinks=true, pdfstartview=FitH, linkcolor=blue, citecolor=blue, urlcolor=blue]{hyperref}
\usepackage{tabstackengine}
\stackMath

\newcommand{\floor}[1]{\left\lfloor #1\right\rfloor}
\newcommand{\set}[1]{\left\{#1\right\}}

\newtheorem{thm}{Theorem}[section]
\newtheorem{lem}[thm]{Lemma}
\newtheorem{prop}[thm]{Proposition}
\newtheorem{cor}[thm]{Corollary}

\theoremstyle{definition}

\newtheorem{rem}[thm]{Remark}

\newcommand{\R}{\mathbb{R}}

\newcommand{\N}{\mathbb{N}}
\newcommand{\E}{\mathbb{E}}

\newcommand{\Z}{\mathbb{Z}}
\newcommand{\F}{\mathbb{F}}

\newcommand{\ve}{\varepsilon}

\title{Avoiding short progressions in Euclidean Ramsey theory}

\author[$\dagger$]{Gabriel Currier}
\author[$\star$]{Kenneth Moore}
\author[$\ddag$]{Chi Hoi Yip}
\affil[$\dagger$,$\star$]{Department of Mathematics\\ University of British Columbia\\ Vancouver  V6T 1Z2\\ Canada}
\affil[$\ddag$]{School of Mathematics\\ Georgia Institute of Technology\\ Atlanta, GA 30332\\ United States}
{
    \makeatletter
    \renewcommand\AB@affilsepx{: \protect\Affilfont}
    \makeatother

    \affil[ ]{Email ids}

    \makeatletter
    \renewcommand\AB@affilsepx{, \protect\Affilfont}
    \makeatother

    \affil[$\dagger$]{currierg@math.ubc.ca}
    \affil[$\star$]{kjmoore@math.ubc.ca}
    \affil[$\ddag$]{cyip30@gatech.edu}
}
\makeatletter
\newcommand{\subjclass}[2][1991]{%
  \let\@oldtitle\@title%
  \gdef\@title{\@oldtitle\footnotetext{#1 \emph{Mathematics subject classification.} #2}}%
}
\newcommand{\keywords}[1]{%
  \let\@@oldtitle\@title%
  \gdef\@title{\@@oldtitle\footnotetext{\emph{Key words and phrases.} #1.}}%
}
\makeatother

\subjclass[2020]{05D10, 52C10, 11B25}
\keywords{Euclidean Ramsey theory, arithmetic progression}

\begin{document}
\maketitle

\begin{abstract}
We provide a general framework to construct colorings avoiding short monochromatic arithmetic progressions in Euclidean Ramsey theory. Specifically, if $\ell_m$ denotes $m$ collinear points with consecutive points of distance one apart, we say that $\E^n \not \to (\ell_r,\ell_s)$ if there is a red/blue coloring of $n$-dimensional Euclidean space that avoids red congruent copies of $\ell_r$ and blue congruent copies of $\ell_s$. We show that $\E^n \not \to (\ell_3, \ell_{20})$, improving the best-known result $\E^n \not \to (\ell_3, \ell_{1177})$ by F\"uhrer and T\'oth, and also establish  $\E^n \not \to (\ell_4, \ell_{14})$ and $\E^n \not \to (\ell_5, \ell_{8})$ in the spirit of the classical result $\E^n \not \to (\ell_6, \ell_{6})$ due to Erd{\H{o}}s et. al. We also show a number of similar $3$-coloring results, as well as $\E^n \not \to (\ell_3, \alpha\ell_{6889})$, where $\alpha$ is an arbitrary positive real number. This final result answers a question of F\"uhrer and T\'oth in the positive.
\end{abstract}

\section{Introduction}

Let $\E^n$ denote the $n$-dimensional Euclidean space, that is, $\R^n$ equipped with the Euclidean norm. Suppose that we have a coloring of $\E^n$ using $r$ colors. The field of Euclidean Ramsey theory concerns itself with what types of configurations (monochromatic, rainbow, etc.) must be present in such a coloring. Much of the current research was introduced and developed by Erd{\H{o}}s, Graham, Montgomery, Rothschild, Spencer, and Straus in a series of papers \cite{E73, E75a, E75b}. We also refer the readers to a nice survey by Graham \cite[Chapter 11]{GOT18}.

One of the most commonly studied configurations is denoted $\ell_m$, which consists of $m$ collinear points with consecutive points of distance one apart. In other words, $\ell_m$ is an $m$-term arithmetic progression with common difference $1$. Let $K_1,\dots,K_r$ be configurations in $\E^n$. We write $\E^n \rightarrow (K_1,\dots,K_r)$ if, for any coloring of $\E^n$ with $r$ colors, there exists a monochromatic (congruent) copy of $K_i$ in color $i$ for some $i$. If there is a coloring where this does not hold, we write $\E^n \not \to (K_1,\dots,K_r)$.

We first recall some positive results related to arithmetic progressions in Euclidean Ramsey theory. Erd{\H{o}}s et. al. \cite[Theorem 8]{E73} proved that $\E^3 \to (T,T)$ for any three-point configuration $T$. Juh\'{a}sz \cite{J79} proved that $\E^2 \to (\ell_2, K)$ for any $K$ with $4$ points, and Csizmadia and T\'{o}th \cite{CT94} showed there exists a configuration $K'$ with $8$ points, such that $\E^2 \not \to (\ell_2,K')$. Tsaturian \cite{T17} showed that $\E^2 \to (\ell_2, \ell_5)$, and Arman and Tsaturian \cite{AT18} showed that $\E^3 \to (\ell_2, \ell_6)$.  As remarked in \cite[Theorem 1.3]{CF19}, Szlam \cite{S01} proved that
there is a constant $c>0$, such that $\E^n \to (\ell_2, \ell_{2^{cn}})$ for all positive integers $n$. Very recently, the authors~\cite{CMY} showed that $\E^2 \to (\ell_3, \ell_3)$, making new progress on the conjecture that $\E^2 \to (T,T)$ for any non-equilateral three-point configuration $T$ by Erd{\H{o}}s et. al. \cite{E75b}. This conjecture is still widely open; see \cite{CMY, E75b, S76} and \cite[Theorem 11.1.4 (a)]{GOT18} for the known cases of this conjecture.

The main contributions of this paper are some new negative results related to arithmetic progressions in Euclidean Ramsey theory. In their seminal paper, Erd{\H{o}}s et. al. \cite{E73} used spherical colorings (that is, where all points of the same distance from the origin have the same color) to prove a few negative results. In particular, they proved that $\E^n \not \to (\ell_6,\ell_6)$ for all positive integers $n$ \cite[Theorem 12]{E73}. Recently, more delicate colorings have been used to prove other negative results. Conlon and Fox \cite{CF19} showed that $\E^n \not \to (\ell_2, \ell_{10^{5n}})$ for all positive integers $n$. Their proof is based on a random periodic coloring. Recently, Conlon and Wu \cite{CW23} showed that $\E^n \not \to (\ell_3, \ell_{10^{50}})$ using spherical coloring with the help of probabilistic methods as well. Very recently, F\"uhrer and T\'oth \cite{FT24} improved their result to $\E^n \not \to (\ell_3, \ell_{1177})$ using a deterministic spherical coloring. We note that these spherical colorings which have been so useful in proving results (in all dimensions) about $\ell_m$ for $m \ge 3$ are generally not available for $\ell_2$; this is because such colorings always contain large monochromatic spheres (in both colors), and will thus contain monochromatic copies of $\ell_2$.

In this paper, we provide a general framework and obtain several new negative results. Our first main result is a further improvement of the result by F\"uhrer and T\'oth \cite{FT24}. Inspired by the result $\E^n \not \to (\ell_6,\ell_6)$ of Erd{\H{o}}s et. al. \cite{E73}, we also prove new negative results related to $\ell_4$ and $\ell_5$.

\begin{thm}\label{thm:l3}
$\E^n \not \to (\ell_3, \ell_{20})$, $\E^n \not \to (\ell_4, \ell_{14})$, and $\E^n \not \to (\ell_5, \ell_{8})$.
\end{thm}

Erd{\H{o}}s et. al. showed that $\E^n \not \to (\ell_4, \ell_4, \ell_4)$ \cite[Theorem 12]{E73}. In the same spirit, we prove the following theorem.

\begin{thm}\label{thm:l3l3}
$\E^n \not \to (\ell_3, \ell_3, \ell_{8})$, $\E^n \not \to (\ell_3, \ell_4, \ell_{7})$, and $\E^n \not \to (\ell_3, \ell_5, \ell_5)$.     
\end{thm}

Our techniques can be used to prove results of a similar flavor for other geometric configurations; see Section~\ref{sec:other}. In particular, we show that there is a red/blue coloring of $\E^n$ that does not contain a red copy of any parallelogram in a $2$-parameter family of parallelograms that includes $\ell_4$ (as a degenerate parallelogram), and does not contain a blue copy of $\ell_{21}$.

More generally, for a positive real number $\alpha$ and an integer $m\geq 2$, we use $\alpha\ell_m$ to denote a set consisting of $m$ collinear points with consecutive points of distance $\alpha$ apart. F\"uhrer and T\'oth \cite[Theorem 2]{FT24} showed $\E^n \not \to (\ell_3, \alpha\ell_{8649})$ for a positive $\alpha$ under some assumptions on $\alpha$. In the concluding remark of their paper, they mentioned that for each $\alpha>0$, their method can be modified to show that there is $m(\alpha)$ finite, such that $\E^n \not \to (\ell_3, \alpha\ell_{m(\alpha)})$. They asked if there is a uniform upper bound on $m(\alpha)$. We answer their question in the positive.

\begin{thm}\label{thm:alpha}
$\E^n \not \to (\ell_3, \alpha\ell_{6889})$ holds for all positive real numbers $\alpha$. 
\end{thm}

\section{A general framework}\label{sec:framework}

In this section, we prove a couple of useful lemmas and then outline the proof of Theorems \ref{thm:l3} and \ref{thm:l3l3}. In general, we will need efficient ways to show there are no short monochromatic progressions in our colorings. The following results, Corollary \ref{cor:ell3} and Proposition \ref{prop:checking} are our tools to do this. To start, the following lemma is a standard application of the law of cosines; see \cite{CW23} and \cite[Lemma 1]{FT24}.

\begin{lem}\label{lem:ell3}
Let $x,y,z\in\mathbb{E}^n$ form a copy of $\alpha\ell_3$. Then we have $$|x|^2-2|y|^2+|z|^2=2\alpha^2.$$
\end{lem}

Next, we deduce two useful corollaries of Lemma~\ref{lem:ell3}.

\begin{cor}\label{cor:ell3}
Let $x,y,z\in\mathbb{E}^n$ form a copy of $\ell_3$ and let $d$ be a positive integer. Then we have 
    \begin{equation*}
        \label{eq:modifiedcoloring}
        \floor{d|x|^2}-2 \floor{d|y|^2}+\floor{ d|z|^2} \in \{2d-1,2d,2d+1\}. 
    \end{equation*}
\end{cor}
\begin{proof}
    Let $d|x|^2 = \floor{d|x|^2} + \ve_x$, where $\ve_x\in [0,1)$, and use similar notation for $y$ and $z$. Then applying Lemma~\ref{lem:ell3},
    $$2d=d(|x|^2-2|y|^2+|z|^2) = \floor{d|x|^2}-2\floor{d|y|^2}+\floor{d|z|^2} + \ve_x -2\ve_y+\ve_z\, .$$
    Since $|\ve_x -2\ve_y+\ve_z|  < 2$, the corollary holds.
\end{proof}

\begin{cor}\label{cor:ellN}
Suppose that $\{x_0, x_1,...,x_{N}\}$ forms a copy of $\alpha{\ell}_{N+1}$ with $N \geq 2$. Let $X_k= |x_k|^2 $. Then 
$$X_k=\alpha^2k^2+(X_1-X_0-\alpha^2)k+X_0$$ holds for all $0 \leq k \leq N$. 
\end{cor}

\begin{proof}
By Lemma~\ref{lem:ell3}, we have $X_{k+2}=2X_{k+1}-X_k+2\alpha^2$ for each $0 \leq k \leq N-2$. By induction on $k$, it follows that 
$$X_k=\alpha^2k^2+(X_1-X_0-\alpha^2)k+X_0$$
holds for all $0 \leq k \leq N$. 
\end{proof}

In view of Corollary~\ref{cor:ellN}, the next proposition leads to a simple algorithm to verify that there is no monochromatic copy of $\ell_{N+1}$ in a given spherical coloring.

\begin{prop}\label{prop:checking}
Let $p,d$ be positive integers, and let $S$ be a subset of $\{0,1,\ldots, p-1\}$. Let $N$ be a positive integer. Then the following two statements are equivalent:
\begin{enumerate}
 \item For all real numbers $b$ and $c$, there is an integer $k$ such that $0 \leq k \leq N$ and $\lfloor d(k^2+bk+c)\rfloor \in S \pmod p$. 
\item For each positive integer $m \leq 2N+1$, and each integer $b_0$ and $c_0$ with $0 \leq b_0 \leq mp$ and $0 \leq c_0 \leq mp$, define
$$
K_i=\bigg\{0 \leq k \leq N: \bigg \lfloor dk^2+\frac{b_0}{m}k+\frac{c_0-i}{m}\bigg\rfloor  \in S \pmod p\bigg\}
$$
for $i \in \{0,1\}$. Then $K_0$ and $K_1$ are nonempty; moreover, the smallest element in $K_i$ is less or equal to the largest element in $K_{1-i}$ for $i \in \{0,1\}$. 
\end{enumerate}
\end{prop}
\begin{proof}

We first show that (1) implies (2). Let $m \leq 2N+1$ be a positive integer, and let $b_0$ and $c_0$ be integers with $0 \leq b_0 \leq mp$ and $0 \leq c_0 \leq mp$. For each $i \in \{0,1\}$, applying (1) with $b=\frac{b_0}{md}$ and $c=\frac{c_0-i}{md}$, it follows that $K_0$ and $K_1$ are nonempty. Let $\ell$ be the smallest element in $K_0$, and let $r$ be the largest element in $K_1$. For the sake of contradiction, assume that $\ell>r$. Take 
$$b=\frac{b_0}{md}-\frac{1}{md(N+1)}, \quad \text{and }  \quad c=\frac{c_0}{md}+\frac{r}{md(N+1)}.$$
Note that for each $0 \leq k \leq N$, we have
$$
d(k^2+bk+c)-\bigg(dk^2+\frac{b_0}{m}k+\frac{c_0}{m}\bigg)=-\frac{k}{m(N+1)}+\frac{r}{m(N+1)}=\frac{r-k}{m(N+1)},
$$
which has absolute value less than $\frac{1}{m}$, while $dk^2+\frac{b_0}{m}k+\frac{c_0}{m}$ is a rational number with denominator $m$. It follows that 
\begin{equation*}
\lfloor d(k^2+bk+c) \rfloor=
\begin{cases}
\lfloor dk^2+\frac{b_0}{m}k+\frac{c_0}{m}\rfloor & \text{ if } 0 \leq k \leq r,\\ 
\lfloor dk^2+\frac{b_0}{m}k+\frac{c_0-1}{m}\rfloor  & \text{ if } r+1 \leq k \leq N.
\end{cases}
\end{equation*}
Thus the assumption $\ell>r$ implies that there is no integer $k$ with $0 \leq k \leq N$, such that $\lfloor d(k^2+bk+c)\rfloor \in S \pmod p$, violating (1). Thus, we must have $\ell \leq r$. Similarly, let $\ell'$ be the smallest element in $K_1$ and $r'$ the largest element in $K_{0}$; we can use 
$$b=\frac{b_0}{md}+\frac{1}{md(N+1)}, \quad \text{and }  \quad c=\frac{c_0}{md}-\frac{r'+1}{md(N+1)}$$ to conduct a similar argument and show that $\ell' \le r'$.

Next, we show that (2) implies (1).  Let $b$ and $c$ be real numbers. Since we are just interested in the value of $\lfloor d(k^2+bk+c)\rfloor$ modulo $p$, we can additionally assume that $0 \leq b< p/d$ and $0 \leq c <p/d$. We shall use rational numbers to approximate $db$ and $dc$. By Dirichlet's approximation theorem (see for example \cite[Theorem 1A]{S80}), we can find a positive integer $m \leq 2N+1$ and an integer $b_0$, such that $$\bigg|bd-\frac{b_0}{m}\bigg|<\frac{1}{m(2N+1)}.$$
Since $0 \leq b<p/d$, we have $0 \leq b_0 \leq mp$. Let $c_0$ be an integer such that 
$$\bigg|cd-\frac{c_0}{m}\bigg| \leq \frac{1}{2m}.$$
Since $0 \leq c <p/d$, we have $0 \leq c_0 \leq mp$.

For each $0 \leq k \leq N$, note that $dk^2+\frac{b_0}{m}k+\frac{c_0}{m}$ is a rational number with denominator $m$, and 
$$
d(k^2+bk+c)-\bigg(dk^2+\frac{b_0}{m}k+\frac{c_0}{m}\bigg)=\bigg(bd-\frac{b_0}{m}\bigg)k+\bigg(cd-\frac{c_0}{m}\bigg),
$$
which has absolute value at most $\frac{k}{m(2N+1)}+\frac{1}{2m}<\frac{1}{m}$. 
It follows that 
\begin{equation*}
\lfloor d(k^2+bk+c) \rfloor=
\begin{cases}
\lfloor dk^2+\frac{b_0}{m}k+\frac{c_0}{m}\rfloor & \text{ if } \big(bd-\frac{b_0}{m}\big)k+\big(cd-\frac{c_0}{m}\big) \geq 0,\\ 
\lfloor dk^2+\frac{b_0}{m}k+\frac{c_0-1}{m}\rfloor  & \text{ if } \big(bd-\frac{b_0}{m}\big)k+\big(cd-\frac{c_0}{m}\big)<0.
\end{cases}
\end{equation*}
Next assume that $bd \leq \frac{b_0}{m}$. Let $\ell$ be the smallest element in $K_0$, and let $r$ be the largest element in $K_1$. By the assumption, $\ell \leq r$. If $\bigg(bd-\frac{b_0}{m}\bigg)r+\bigg(cd-\frac{c_0}{m}\bigg) < 0$, then by the definition of the set $K_1$, we have
$$
\lfloor d(r^2+br+c) \rfloor=\lfloor dr^2+\frac{b_0}{m}r+\frac{c_0-1}{m}\rfloor \in S \pmod p
$$
and we are done. Otherwise, $r \geq \ell$ implies that
$$
0 \le \bigg(bd-\frac{b_0}{m}\bigg)r+\bigg(cd-\frac{c_0}{m}\bigg) \leq \bigg(bd-\frac{b_0}{m}\bigg)\ell+\bigg(cd-\frac{c_0}{m}\bigg)
$$
and thus
$$
\lfloor d(\ell^2+b\ell+c) \rfloor=\lfloor d\ell^2+\frac{b_0}{m}\ell+\frac{c_0}{m}\rfloor \in S \pmod p
$$
by the definition of the set $K_0$ and we are also done.

Finally, the assumption that the smallest element in $K_1$ is less or equal to the largest element in $K_0$ can be used to handle the case $bd >\frac{b_0}{m}$ in a similar way.
\end{proof}

\subsection{Outline of the proofs}

In general, we want to color $\E^n$ red and blue such that there are no red $\ell_i$ and blue $\ell_j$, where $i \in \{3,4,5\}$ and $j$ is as small as possible. To do so, we choose positive integers $p$ and $d$, and a subset $S \subset \{0,1,\dots,p-1\}$, and color the following points red
\begin{equation}
\label{eq:reds}
\mathcal{R}:=\{x\in\mathbb{E}^n\ : \lfloor d|x|^2 \rfloor \in S \pmod p\},
\end{equation}
and let the blue set be $\mathcal{B} := \E^n \setminus \mathcal{R}$. We can judiciously choose $S$ (generally, as large as possible) so that there are no red $\ell_i$; since $i$ is small, a quick computational search finds all valid $S$, and it is easy to verify by hand that there are no red $\ell_i$ for an individual choice of $S$. In particular, Corollary~\ref{cor:ell3} tells us how to check whether a red coloring via \eqref{eq:reds} will contain a red $\ell_3$. %; just check whether $s_1-2s_2+s_3\in \set{2d-1,2d,2d+1}$ for triples of elements of $S$. 
Since $j$ is generally larger, it is a non-trivial task to show that there are no blue $\ell_j$. However, we note that by Corollary \ref{cor:ellN}, for any $\ell_j$ given by points $\{x_0,\dots,x_{j-1}\}$ we have that  $d|x_k|^2 = d(k^2+bk+c)$ for some real-valued coefficients $b$ and $c$. If $j$ is sufficiently large, we use Proposition \ref{prop:checking} to verify that, indeed, some $x_k$ is colored red.
\begin{comment}
\begin{rem}    
When considering the optimality of this technique, it is interesting to consider the case when Proposition \ref{prop:checking} fails; that is, when there exist real numbers $b,c$ and a positive integer $N$ such that $\lfloor dk^2 + bk + c \rfloor \notin S \pmod p$ for all $0 \le k \le N$. We note that in this case there \emph{will} generally be a monochromatic blue $\ell_{N+1}$ in the above coloring. To see this, we choose points $x_0,x_1 \in \E^n$ such that $d|x_0|^2 \equiv c$ and $d|x_1|^2 \equiv d+ b +c$ in $(\R / p \Z)$, and also that $|x_0-x_1| = 1$. This is always possible if we choose $x_0,x_1$ sufficiently far from the origin. Let $d|x_0|^2=c'$ and $d|x_1|^2=d+b'+c'$, where $b'-b$ and $c'-c$ are integer multiples of $p$. Let $x_k=k(x_1-x_0)+x_0$ for each $2 \leq k \leq N$, so that $\{x_0, \dots, x_{N}\}$ forms a copy of $\ell_{N+1}$. Note that, for each $0 \leq k \leq N$, we have $d|x_k|^2 = d(k^2+b'k+c')$ by Corollary~\ref{cor:ellN}, and it follows from the assumption that $\lfloor d|x_k|^2 \rfloor \equiv \lfloor dk^2 + bk + c \rfloor \notin S \pmod p$ \textcolor{red}{this is not quite correct: I guess we can only show that $\lfloor d|x_k|^2 \rfloor \equiv \lfloor dk^2 + dbk +dc \rfloor \notin S \pmod p$}. We then conclude that the points $x_0, \dots, x_{N}$ are colored blue. Thus, the smallest $N$ given to us by Proposition \ref{prop:checking} will be the optimal value for this coloring.
\end{rem}
\end{comment}
\begin{rem}    
When considering the optimality of this technique, it is interesting to consider the case when Proposition \ref{prop:checking} fails; that is, when there exist real numbers $b,c$ and a positive integer $N$ such that $\lfloor d(k^2 + bk + c) \rfloor \notin S \pmod p$ for all $0 \le k \le N$. We note that in this case there \emph{will} generally be a monochromatic blue $\ell_{N+1}$ in the above coloring. To see this, we choose points $x_0,x_1 \in \E^n$ such that $d|x_0|^2 \equiv dc$ and $d|x_1|^2 \equiv d+ db +dc$ in $(\R / p \Z)$, and also that $|x_0-x_1| = 1$. This is always possible if we choose $x_0,x_1$ sufficiently far from the origin. Let $d|x_0|^2=c'$ and $d|x_1|^2=d+b'+c'$, where $b'-db$ and $c'-dc$ are integer multiples of $p$. Let $x_k=k(x_1-x_0)+x_0$ for each $2 \leq k \leq N$, so that $\{x_0, \dots, x_{N}\}$ forms a copy of $\ell_{N+1}$. Note that, for each $0 \leq k \leq N$, we have $|x_k|^2 = k^2+\frac{b'}{d}k+\frac{c'}{d}$ by Corollary~\ref{cor:ellN}, and it follows from the assumption that $\lfloor d|x_k|^2 \rfloor \equiv \lfloor dk^2 + b'k + c' \rfloor \equiv \lfloor d(k^2 + bk + c) \rfloor \notin S \pmod p$. We then conclude that the points $x_0, \dots, x_{N}$ are colored blue. Thus, the smallest $N$ given to us by Proposition \ref{prop:checking} will be the optimal value for this coloring.
\end{rem}

\section{Proof of Theorem~\ref{thm:l3} and Theorem~\ref{thm:l3l3}}
In this section, we use certain choices of $p$, $d$, and $S$ to prove Theorem~\ref{thm:l3} and Theorem~\ref{thm:l3l3}. Checking our proof requires a simple computer program, which can be found at \cite{ResourcesProgression}; furthermore, this program is implemented using only integers to avoid floating point errors. 

\subsection{Proof of Theorem~\ref{thm:l3}} 

The choices that lead to Theorem~\ref{thm:l3} are the following. Note that any translate of a valid $S$ works just as well.

\begin{itemize}
    \item $\E^n \not \to (\ell_3, \ell_{20})$: Let $p=29$, $d=7$, and $S=\set{0,1,2,3,4,5,6}$.
    \item $\E^n \not \to (\ell_4, \ell_{14})$: Let $p=29$, $d=10$, and $S=\{0,1,2,3,4,5,6,7,8\}$.
    \item $\E^n \not \to (\ell_5, \ell_{8})$: Let $p=5$, $d=2$, and $S=\set{0,1}$.
\end{itemize}

We will consider the $(\ell_3, \ell_{20})$ result from above to explain what our algorithm does specifically. Recall the definition of our coloring, whose red points are given in equation \eqref{eq:reds}. First, we have to check that at least one point in every $\ell_{20}$ is red. By Corollary \ref{cor:ellN}, it suffices to check that for every real numbers $b$ and $c\in [0,29)$, there is a $k\in \set{0,...\, ,19}$ such that $\lfloor 7(k^2+bk+c)\rfloor \in \set{0,...\, , 6} \pmod {29}$. If we let $K_0,K_1$ be defined as in Proposition~\ref{prop:checking}, then to finish the claim, we need only check that each $K_i$ is non-empty, and that $\max\{K_0\} \geq \min\{K_1\}$ and $\max\{K_1\} \geq \min\{K_0\}$. Computing $K_i$ involves only checking a relatively small number of integer inputs to polynomials with rational coefficients that have denominator at most $39$. Thus, the computation is finite and can be easily checked on a regular computer. 

Next, we need to be sure that no triple $x_1$, $x_2$, $x_3$ forming an $\ell_3$ is entirely red; this would be the case if $\floor{7|x_i|^2}\in S$ for $i=0,1,2$. To check this, we write the complement set $S'=\{0,1,\ldots, p-1\}\setminus S$, and run our algorithm on the set $S'$ with the same $p$ and $d$ values in exactly same manner as in the previous paragraph.

The proofs of $\E^n \not \to (\ell_4, \ell_{14})$ and $\E^n \not \to (\ell_5, \ell_{8})$ are almost identical. The only notable changes are to $p,d$ and $S$.

\begin{rem}
We have run a reasonably exhaustive search on different $p,d$, and $S$ that might be useful for these purposes. Thus, it is somewhat curious that in each case, we settled on sets of consecutive integers. We do not have a good explanation for this at the moment. However, for the interested reader, we record below the $p,d$, and $S$ that we checked for these purposes.
\begin{itemize}
    \item $i=3$: we checked all values $3\leq p\leq 45$, $1\leq d \leq \frac{p}{2}$, and all sets $S$ where $|S| \leq 10$. 
    \item $i=4$: we checked all values $3\leq p \leq 35$, $1\leq d \leq \frac{p}{2}$, and all sets $S$ where $|S| \leq 10$. 
    \item $i=5$: we checked all values $3\leq p \leq 29$, $1\leq d \leq \frac{p}{2}$, and all sets $S$ where $|S| \leq 10$.
\end{itemize}
\end{rem}

\begin{rem}
This framework can recover the result $\E^n \not \to (\ell_6,\ell_6)$ of Erd{\H{o}}s et. al. \cite{E73} as well; simply set $p=12, d=1$, and $S=\{0,1,2,3,4,5\}$ \cite{E73}. 
\end{rem}

\subsection{Proof of Theorem~\ref{thm:l3l3}}
We use a similar strategy to prove Theorem~\ref{thm:l3l3}. We now have 3 colors, say red, green, and blue. We shall choose positive integers $p$ and $d$, and two disjoint subsets $S$ and $T$ of $\{0,1,\ldots, p-1\}$. We define the red point set $\mathcal{R}$ and the green point set $\mathcal{G}$ as follows:
$$
\mathcal{R}:=\{x\in\mathbb{E}^n\ : \lfloor d|x|^2 \rfloor \in S \pmod p\}, \quad \mathcal{G}:=\{x\in\mathbb{E}^n\ : \lfloor d|x|^2 \rfloor \in T \pmod p\}.
$$
Then the blue point set $\mathcal{B}$ is
$$
\mathcal{B}:=\{x\in\mathbb{E}^n\ : \lfloor d|x|^2 \rfloor \not \in S \cup T \pmod p\},
$$
and we can apply Proposition~\ref{prop:checking}, as in the proof of Theorem~\ref{thm:l3} but replacing $S$ with $S \cup T$, to check that there is no blue progression.

\begin{enumerate}
    \item $\E^n \not \to (\ell_3, \ell_3, \ell_{8})$: Let $p=10$, $d=2$, $S = \set{0, 1}$, and $T=\set{5, 6}$.
    \item $\E^n \not \to (\ell_3, \ell_4, \ell_{7})$: Let $p=10$, $d=2$, $S = \set{0, 1}$, and $T=\set{4, 5, 6}$.
    \item $\E^n \not \to (\ell_3, \ell_5, \ell_5)$: Let $p=8$, $d=1$, $S=\{0,4\}$, and $T=\{5, 6, 7\}$.  
\end{enumerate}

\begin{rem}
By setting $p=3,d=2, S=\{0\},$ and $T=\{1\}$, we recover the $(\ell_4, \ell_4, \ell_4)$ result by Erd{\H{o}}s et. al. \cite{E73}. 
\end{rem}

\subsection{Other configurations}\label{sec:other}
Conlon and Wu \cite[Conjecture 4.2]{CW23} conjectured that for every non-spherical set $X$, there exists a natural number $m$ such that $\E^n \not \to (X,\ell_m)$ for all $n$. Our techniques can be modified to prove this conjecture for some sets other than arithmetic progressions, such as regular polygons with their centers, or parallelograms. We illustrate the latter example here. Note that this conjecture was proven recently (after the release of the present manuscript) by Conlon and F\" uhrer in \cite{CF24}; however, the present results are still of interest because they provide strong bounds on the value of $m$, and avoid larger classes of configurations in red.
\begin{lem}
    \label{lem:parallelogram}
    Let $x_1,x_2, x_3,x_4\in \E^n$ form a parallelogram, with diagonal lengths $|x_1-x_2|=\ \alpha$ and $|x_3-x_4|=\ \beta$. Then we have
    \begin{equation}
    \label{eq:parallelogram}
        |x_1|^2+|x_2|^2-|x_3|^2-|x_4|^2 = \frac12 (\alpha^2-\beta^2).
    \end{equation}
\end{lem}
\begin{proof}
Consider Figure~\ref{fig:parallelogram} below. Let $X_0=\frac12|x_1+x_2|$ be the norm of the centroid.
 
\begin{figure}[H]
    \centering
    \includegraphics[scale=0.55]{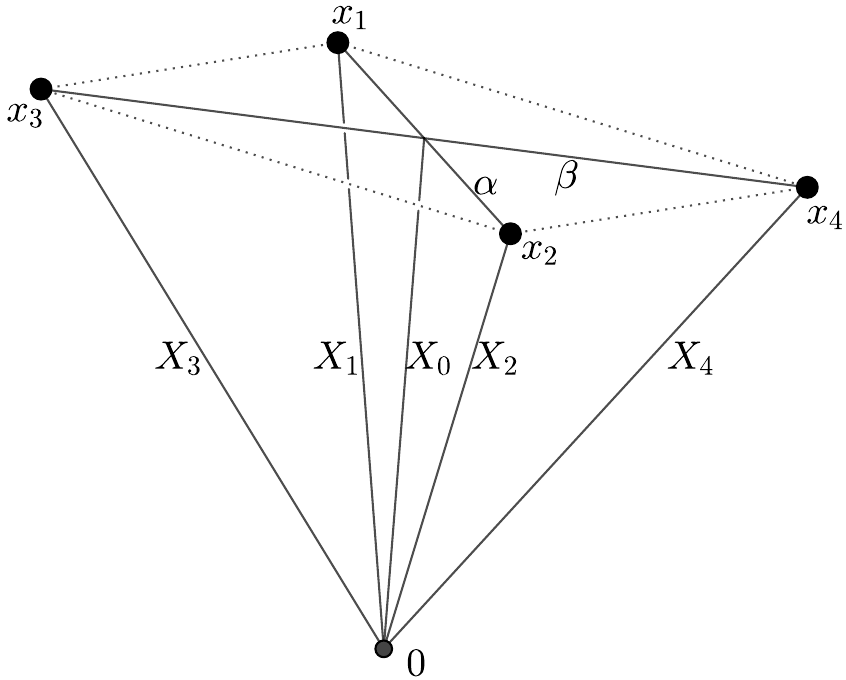}
    \caption{Four points forming a parallelogram}
    \label{fig:parallelogram}
\end{figure}
We know $|x_1-x_2|=\alpha$, $|x_3-x_4|=\beta$, and $|x_1+x_2|=|x_3+x_4|=2X_0$. This gives us two equations,
\begin{align*}
    2(|x_1|^2+|x_2|^2)=|x_1+x_2|^2+|x_1-x_2|^2&=4X_0^2+\alpha^2\, ,
    \\
    2(|x_3|^2+|x_4|^2)=|x_3+x_4|^2+|x_3-x_4|^2&=4X_0^2+\beta^2\, .
\end{align*}
Subtracting these equations yields equation \eqref{eq:parallelogram}.
\end{proof}
We next have a corollary similar to Corollary~\ref{cor:ell3} when $\frac12 (\alpha^2-\beta^2)$ is a positive integer.
\begin{cor}
\label{cor:parallelogram}
Let $x_1,x_2, x_3,x_4\in \E^n$ form a parallelogram, with diagonal lengths $\alpha$ and $\beta$ such that $\gamma:=\frac12 (\alpha^2-\beta^2)$ is a positive integer. Then we have
    \begin{equation*}
        \floor{d|x_1|^2}+\floor{d|x_2|^2}-\floor{d|x_3|^2}-\floor{d|x_4|^2} \in \set{d\gamma-1,d\gamma,d\gamma+1}.
    \end{equation*}
\end{cor}

%We can now use this in conjunction with Proposition~\ref{prop:checking} to search for colorings that avoid parallelograms in red and progressions in blue. For a particular case, we consider $\gamma=4$. Denote by $\mathcal{P}_4$ the collection of parallelograms with diagonals satisfying $\frac12 (\alpha^2-\beta^2)=4$. This is not a congruence class of one point set, rather, it is the congruence classes from \emph{a $2$-parameter family} of parallelograms, and in particular, $\ell_4\in \mathcal{P}_4$ with $\alpha=3$ and $\beta=1$. We can create a coloring that avoids this \emph{entire family} in red. Indeed, by setting $d=2$, $p=17$, $S=\set{0,1,2,3 }$, we have $\E^n\not\to (\mathcal{P}_4, \ell_{21})$, which is only slightly worse than the bound for $\ell_4$ alone from Theorem~\ref{thm:l3}.
We can now use this in conjunction with Proposition~\ref{prop:checking} to search for colorings that avoid parallelograms in red and progressions in blue. For a particular case, we consider $\gamma=2$. Denote by $\mathcal{P}_2$ the collection of parallelograms with diagonals satisfying $\frac12 (\alpha^2-\beta^2)=2$. This is not a congruence class of one point set, rather, it is the congruence classes from \emph{a $2$-parameter family} of parallelograms, and in particular, $\ell_3\in \mathcal{P}_2$ with $\alpha=2$ and $\beta=0$. We can create a coloring that avoids this \emph{entire family} in red. Indeed, using one of the same colorings from the proof of Theorem \ref{thm:l3}, we show that $\E^n\not\to (\mathcal{P}_2, \ell_{20})$ with exactly the same bound we had for $\ell_3$ alone.

Similarly, we use $\mathcal{P}_{\gamma}$ to denote the collection of parallelograms with diagonals satisfying $\frac12 (\alpha^2-\beta^2)=\gamma$. Note also that $\ell_4 \in \mathcal{P}_4$, which we can avoid for only a slightly worse bound than we had in Theorem~\ref{thm:l3}. We report results for $\gamma \in \{1,2,3,4\}$ below:
\begin{itemize}
    \item $\E^n\not\to (\mathcal{P}_1, \ell_{18})$:  $p=31,d=8, S=\{0, 1, 2, 3, 16, 17, 18\}$.
    \item $\E^n\not\to (\mathcal{P}_2, \ell_{20})$:  
$p=29,d=7, S=\{0, 1, 2, 3, 4, 5, 6\}$.
\item $\E^n\not\to (\mathcal{P}_3, \ell_{19})$: $p=25,d=4, S=\{0, 1, 2, 3, 4, 5\}$.
\item $\E^n\not\to (\mathcal{P}_4, \ell_{21})$: $p=17$, $d=2$, $S=\set{0,1,2,3 }$
\end{itemize}

\begin{rem}
More generally, suppose we have a family of sets whose members satisfy a linear relation in their points' square norms with rational coefficients. Additionally, assume that this linear relationship is non-homogeneous, and the sum of the coefficients of the non-constant terms should not add up to the constant term (otherwise, the configuration is spherical). Lemmas \ref{lem:ell3} and \ref{lem:parallelogram} are examples of such relations, and any such family has an analogue of Corollaries \ref{cor:ell3} and \ref{cor:parallelogram}. Therefore our methods can be extended to those families as well. It seems plausible, even, that this could be done in the case when the coefficients are non-rational, following a coloring of the type used in \cite{CF24}. However, this would seem to require a new version of Proposition \ref{prop:checking}, and it is not clear that the result would be computationally tractable.
\end{rem}

\section{Proof of Theorem~\ref{thm:alpha}}

In this section, we will prove $\E^n \not \to (\ell_3, \alpha\ell_{6889})$ for any real number $\alpha > 0$. Our proof is based on a refinement of the ideas of F\"uhrer and T\'oth \cite{FT24}, together with a few new ingredients. Some of the arguments presented below have appeared in \cite{FT24}, but we present a self-contained proof for the sake of completeness.

Throughout the section, $p$ is used to denote an odd prime. Furthermore, we use $\F_p$ to denote the finite field with $p$ elements and $\F_p^*=\F_p \setminus \{0\}$. For convenience, sometimes we identity $\F_p$ with the set $\{0,1,\ldots, p-1\}$ without explicitly saying so. In the following constructions, we will use squares in $\F_p$ frequently. We write $R_p=\{x^2: x \in \F_p\}$ and $R_p^*=\{x^2: x \in \F_p^*\}=R_p \setminus \{0\}$. For a real number $\beta$ and a set $I$, we follow the standard notation $\beta \in I \pmod p$, meaning that there is an integer $m$ such that $\beta \in I+pm=\{t+pm: t \in I\}$.

As in previous sections, we will need some tools to show how our coloring avoids short red progressions and long blue progressions. However, instead of trying to avoid $\ell_3$ and $\alpha\ell_{6889}$, we will try to avoid $\alpha_{red} \ell_3$ and $\alpha_{blue} \ell_{6889}$, where $\alpha = \alpha_{blue}/\alpha_{red}$ and $\alpha_{blue},\alpha_{red}$ are judiciously chosen, and finish the proof by scaling. Similar to the framework described in Section~\ref{sec:framework}, we will specify a prime $p$ and a set $S \subset \F_p$, and color a point $x \in \E^n$ red if and only if  $\lfloor |x|^2 \rfloor \in S$ modulo $p$. 

Our first lemma concerns some specific choices of $p$ and $S$, and will have two parts. The first says that all shifts of the squares (and non-squares) modulo $p$ will intersect $S$, and thus contain red points. The second will show that as long as $\alpha_{red}^2 \in [1,1.5]$ modulo $p$, then under this coloring we are avoiding red $\alpha_{red}\ell_3$. This can be seen as an approximate version of Corollary \ref{cor:ell3}.

\begin{lem}\label{lem:red}
For each of the following prime $p$ and subset $S \subset \mathbb{F}_p$, 
\begin{itemize}
    \item $p=47, S=\{0,5,10,15,20\}$
    \item $p=59, S=\{0, 5, 10, 15, 20, 25\}$
    \item $p=67, S=\{0, 5, 10, 15, 20, 25, 30\}$
    \item $p=71, S=\{0, 5, 10, 15, 20, 25, 30\}$
    \item $p=73, S=\{0, 5, 10, 15, 20, 25, 30\}$
 \item $p=79, S=\{0, 5, 10, 15, 20, 25, 30\}$
 \item $p=83, S=\{0, 5, 10, 15, 20, 25, 30,35\}$
\end{itemize}

we have:
\begin{enumerate}
    \item For each $b \in \F_p^*$ and $c \in \F_p$, we have
$$
(bR_p+c) \cap S\neq \emptyset.
$$
\item If $\alpha^2 \in [1,1.5] \pmod p$, then the red-blue coloring of $\E^n$ with red points $$\mathcal{R}:=\{x\in\mathbb{E}^n\ : \lfloor |x|^2 \rfloor \in S \pmod p\}$$ does not contain a red copy of $\alpha{\ell}_3$.
\end{enumerate} 

\end{lem}
\begin{proof}
(1) Easy to verify using computers. The code can be found at \cite{ResourcesProgression} as well. Note that it suffices to verify the statement for $b=1$ and a fixed quadratic non-residue modulo $p$. 

(2) Suppose otherwise that $x, y, z$ form a red copy of $\alpha{\ell}_3$, then Lemma~\ref{lem:ell3} implies that $|x|^2-2|y|^2+|z|^2=2\alpha^2 \in [2,3] \pmod p$. Note that $|x|^2+|z|^2-2<\lfloor |x|^2\rfloor+\lfloor |z|^2 \rfloor \leq |x|^2+|z|^2$ and $-2|y|^2\leq -2\lfloor |y|^2\rfloor<-2|y|^2+2$. It follows that $\lfloor |x|^2\rfloor-2\lfloor |y|^2\rfloor+\lfloor |z|^2 \rfloor \in (0,5) \pmod p$, or equivalently,
$$\lfloor |x|^2\rfloor-2\lfloor |y|^2\rfloor+\lfloor |z|^2 \rfloor\in \{1,2,3,4\} \pmod p.$$
However, since $x,y,z \in \mathcal{R}$, based on the property of the set $S$, we have 
$$
\lfloor |x|^2\rfloor-2\lfloor |y|^2\rfloor+\lfloor |z|^2 \rfloor \equiv 5k \pmod p,
$$
where $k$ is an integer such that $|5k|\leq p-7$, a contradiction.
\end{proof}

\begin{rem}\label{rem:59}
When $p=59$ and $S=\{0, 5, 10, 15, 20, 25\}$, the stronger statement $(bR_p^*+c) \cap S \neq \emptyset$ holds for each $b \in \F_p^*$ and $c \in \F_p$. This observation will be needed in later discussions.
\end{rem}

Next, we will show how to avoid long $\alpha_{blue}\ell_M$, where $\alpha_{blue}^2$ is close to an integer. If $\{x_0,\dots,x_{M-1}\}$ is such a progression, the proof proceeds by showing that $\{\lfloor |x_i|^2 \rfloor: 0\leq i \leq M-1\}$ will always contain a shift of the squares or non-squares, if $M$ is sufficiently large. Combined with the previous lemma, we see that $\alpha_{blue}\ell_M$ will always contain a point colored red.

\begin{lem}\label{lem:blue}
Let $p$ be an odd prime. Let $\alpha$ be a real number such that $\alpha^2=b+\epsilon$ with $b$ an integer not divisible by $p$ and $0\leq \epsilon \leq (4p^5)^{-1}$. Suppose that 
$\{x_0, x_1,...,x_{M-1}\}$ forms a copy of $\alpha{\ell}_{M}$, with $X_k=|x_k|^2$. Assume that $M=p^2$ if $\epsilon=0$, and assume that $M=2p^2-2p+1$ if $\epsilon>0$. If $\epsilon=0$, then there exists $c\in \F_p$, such that 
$$bR_p+c\subseteq \{\lfloor X_k \rfloor: 0 \leq k \leq M-1 \}$$ 
as two subsets of $\F_p$. If $\epsilon>0$, then there exists $c\in \F_p$, such that 
$$bR_p^*+c\subseteq \{\lfloor X_k \rfloor: 0 \leq k \leq M-1 \}$$ 
 as two subsets of $\F_p$. 
\end{lem}

\begin{proof}
By Corollary~\ref{cor:ellN}, for each $0 \leq k \leq M-1$, we have
$X_k=\alpha^2k^2+(X_1-X_0-\alpha^2)k+X_0$. Let $\beta=X_1-X_0-\alpha^2$. By Dirichlet's approximation theorem (see for example \cite[Theorem 1A]{S80}), we can find a positive integer $d \leq p-1$ and an integer $\gamma $, such that $|d\beta-\gamma |\leq 1/p$. Let $\delta=\beta-\gamma /d$; then we have $|\delta|\leq 1/dp$. 

Since $p$ is an odd prime, $1\leq d \leq p-1$, and $b$ is not divisible by $p$, we can find a nonnegative integer $r<p$ such that $2drb \equiv -\gamma  \pmod p$. 

For each $0 \leq j \leq (M-1-r)/d$, let $Y_j=X_{dj+r}$. We compute
\begin{align*}
Y_j
&=\alpha^2(dj+r)^2+\beta (dj+r)+X_0\\
&=(b+\epsilon)(dj+r)^2+(\gamma /d+\delta)(dj+r)+X_0\\
&=b(d^2j^2+2drj+r^2)+\gamma j+ (\gamma r/d+ X_0)+\epsilon (dj+r)^2 +\delta( dj+r)\\
&=bd^2j^2+(2drb+\gamma )j +(br^2+\gamma r/d+X_0+\delta r)+\epsilon (dj+r)^2 +\delta dj.
\end{align*}
Let $$
Z_j:= (br^2+\gamma r/d+X_0+\delta r)+\epsilon (dj+r)^2 +\delta dj.
$$

Next, we consider two cases according to whether $\epsilon=0$. 

(1) Assume that $\epsilon=0$. Since $M=p^2$ and $r,d\leq p-1$, we have $(M-1-r)/d \geq p$. Thus, to show that there exists $c\in \F_p$ such that
$$bR_p+c\subseteq \{\lfloor X_k \rfloor: 0 \leq k \leq M-1 \},$$ 
it suffices to show that there exists $c\in \F_p$ such that
$$bR_p+c\subseteq \{\lfloor Y_j \rfloor: 0 \leq j \leq p \}.$$ 

Recall that $r$ is chosen so that $2drb+\gamma  \equiv 0 \pmod p$, and observe that 
$$
bR_p=\{bd^2i^2: 0 \leq i \leq (p-1)/2\}=\{bd^2i^2: (p+1)/2 \leq i \leq p\}
$$
as subsets of $\F_p$. Therefore, it suffices to show that $Z_j$ has the same integer part for all $j \in [0, (p-1)/2]$, or $Z_j$ has the same integer part for all $j \in [(p+1)/2, p]$. However, this follows from the observation that $Z_j=(br^2+\gamma r/d+X_0+\delta r)+\delta dj$ is monotone in $j$, and $|Z_{j+p}-Z_j|=|\delta| dp\leq 1$. 

(2) Assume that $\epsilon>0$. By the assumption, we have $\epsilon\leq (4p^5)^{-1}$. Since $M=2p^2-2p+1$ and $r,d\leq p-1$, we have $(M-1-r)/d \geq 2p-1$. Thus, to show that there exists $c\in \F_p$ such that
$$bR_p^*+c\subseteq \{\lfloor X_k \rfloor: 0 \leq k \leq M-1 \}, $$
it suffices to show that there exists $c\in \F_p$ such that
$$bR_p^*+c\subseteq \{\lfloor Y_j \rfloor: 0 \leq j \leq 2p-1 \}. $$

Note that in this case, the sequence $Z_j$ is quadratic in $j$ and thus it changes its monotonicity at most once. Therefore, there is $s \in \{0,1\}$, such that $Z_j$ is monotone for $j \in [sp, (s+1)p-1]$. Also, note that 
$$
bR_p^*=\{bd^2(sp+i)^2: 1 \leq i \leq (p-1)/2\}=\{bd^2(sp+i)^2: (p+1)/2 \leq i \leq p-1\},
$$
as subsets of $\F_p$. Therefore, it suffices to show that $Z_j$ has the same integer part for all $j \in [sp+1, sp+(p-1)/2]$, or $Z_j$ has the same integer part for all $j \in [sp+(p+1)/2, (s+1)p-1]$. However, this follows from the observation that $Z_j$ is monotone in $j$ for $j \in [sp,(s+1)p-1]$, and the fact that 
\[|Z_{(s+1)p-1}-Z_{sp}|\leq |\delta| d(p-1) +|\epsilon|M^2< \frac{p-1}{p}+4p^4|\epsilon|\leq \frac{p-1}{p}+\frac{1}{p}=1. \qedhere
\]
\end{proof}

Lemma~\ref{lem:red} and Lemma~\ref{lem:blue} imply the following corollary.

\begin{cor}\label{cor:redblue}
Let $p \in \{47,59,67,71,73, 79,83\}$. Assume that $\alpha_{red}^2 \in [1,1.5] \pmod p$ and $\alpha_{blue}^2=b+\epsilon$, where $b$ is an integer not divisible by $p$ and $0\leq \epsilon \leq (4p^5)^{-1}$. Let $\alpha=\alpha_{blue}/\alpha_{red}$. Then we have
\begin{enumerate}
    \item If $\epsilon=0$, then $\E^n \not \to (\ell_3, \alpha \ell_{p^2})$; 
    \item if $\epsilon>0$ and $p=59$, then $\E^n \not \to (\ell_3, \alpha \ell_{2p^2-2p+1})$.
\end{enumerate}
\end{cor}

\begin{proof}
Let $M=p^2$ if $\epsilon=0$, and $M=2p^2-2p+1$ if $\epsilon>0$. By scaling, it suffices to show that $\E^n \not \to (\alpha_{red}\ell_3, \alpha_{blue} \ell_{M})$. 

Using the coloring given in Lemma~\ref{lem:red}, it does not contain a red copy of $\alpha_{red}{\ell}_3$. Moreover, for each $c \in \F_p$, we have $(bR_p+c) \cap S \neq \emptyset$. When $p=59$, by Remark~\ref{rem:59}, we have the stronger statement that for each $c \in \F_p$, we have $(bR_p^*+c) \cap S \neq \emptyset$. Note that the set of blue points is given by
$$\mathcal{B}=\{x\in\mathbb{E}^n\ : \lfloor |x|^2 \rfloor \notin S \pmod p\}=\{x\in\mathbb{E}^n\ : \lfloor |x|^2 \rfloor \in \F_p \setminus S \pmod p\}.$$
Thus, Lemma~\ref{lem:blue} implies the given coloring does not contain a blue copy of $\alpha_{blue}{\ell}_M$. We conclude that $\E^n \not \to (\alpha_{red}\ell_3, \alpha_{blue} \ell_{M})$, as required.
\end{proof}

Now, we can put our tools together. The trick here will be to find $\alpha_{red}$ and $\alpha_{blue}$ that satisfy the conditions of the previous lemmas, where $\alpha = \alpha_{blue}/\alpha_{red}$. This will involve some casework; we start this in the following proposition and then finish below in the proof of Theorem \ref{thm:alpha}.

\begin{prop}\label{prop:rational}
Let $p \in \{47,59,67,71,73, 79,83\}$. Let $\alpha>0$. The following statements hold:
\begin{enumerate}
    \item If $\alpha^2$ is irrational, then we have $\E^n \not \to (\ell_3, \alpha \ell_{p^2})$.
    \item If $\alpha^2=a/b$ is a rational number in simplest form with $a,b>0$, such that $p \nmid b$ or $a>4p/3$, then we have $\E^n \not \to (\ell_3, \alpha \ell_{p^2})$.
\end{enumerate}
\end{prop}

\begin{proof}
In each of the following cases, we construct $\alpha_{red}$ and $\alpha_{blue}$ satisfying the assumptions in Corollary~\ref{cor:redblue} with $\epsilon=0$ and $\alpha=\alpha_{blue}/\alpha_{red}$, so the statement of the proposition follows from Corollary~\ref{cor:redblue}(1).

(1) Assume that $\alpha^2$ is irrational.  Since $1/\alpha^2$ is irrational, the sequence $(m/\alpha^2)_{m \in \N}$ is uniformly distributed modulo $1$ by Weyl's criterion \cite[Example 2.1]{KN74}. It follows that there is a positive integer $m$ such that
$$
\frac{pm+1}{p\alpha^2}=\frac{m}{\alpha^2}+\frac{1}{p\alpha^2} \in \bigg[\frac{1}{p}, \frac{1.5}{p}\bigg] \pmod 1,
$$
and thus $(pm+1)/\alpha^2 \in [1,1.5] \pmod p$. Let $\alpha_{blue}=\sqrt{pm+1}$ and $\alpha_{red}=\sqrt{pm+1}/\alpha$. Then clearly $\alpha=\alpha_{blue}/\alpha_{red}$. Note that $\alpha_{red}^2 =(pm+1)/\alpha^2 \in [1,1.5] \pmod p$, and $\alpha_{blue}^2=pm+1$ is an integer and it is not divisible by $p$. Thus, Corollary~\ref{cor:redblue} (1) implies that $\E^n \not \to (\ell_3, \alpha \ell_{p^2})$.

(2) We consider three different cases according to whether $a$ and $b$ are divisible by $p$. 

\textbf{Case 1: $p \nmid b$ and $p \nmid a$.} Since $p \nmid b$,  we can find a positive integer $m$ such that $bm \equiv 1 \pmod p$. Let $\alpha_{blue}=\sqrt{am}$ and $\alpha_{red}=\sqrt{bm}$. Then clearly $\alpha=\alpha_{blue}/\alpha_{red}$. Note that $\alpha_{red}^2 =bm \equiv 1 \pmod p$, and $\alpha_{blue}^2=am$ is an integer not divisible by $p$. 

\textbf{Case 2: $p \nmid b$ and $p \mid a$.} Since $\gcd(b, ap)=1$, we can find a positive integer $m$ such that $bm \equiv 1 \pmod {ap}$. Let $\alpha_{blue}=\sqrt{(a+1)m}$ and $\alpha_{red}=\sqrt{(a+1)bm/a}$; then clearly $\alpha=\alpha_{blue}/\alpha_{red}$. Note that 
$$
\alpha_{red}^2 =\frac{(a+1)bm}{a}=bm+\frac{bm}{a} \equiv 1+\frac{1}{a} \in [1,1.5] \pmod p,
$$
and $\alpha_{blue}^2=(a+1)m$ is an integer not divisible by $p$. 

\textbf{Case 3: $p \mid b$ and $a>4p/3$.} Let $r=\lceil a/p \rceil$. We claim that
\begin{align}
\frac{1}{p}\leq \frac{r}{a} \leq \frac{1.5}{p}.\label{line:ra}
\end{align}
Indeed, if $a \geq 2p$, then it is clear that $pr<a+p\leq 1.5a$; if $4p/3<a<2p$, then we have $r=2$ and thus $pr=2p<1.5a$.

Since $\gcd(a,b)=1$, we can find a positive integer $m$ such that $bm \equiv 1 \pmod a$. Let $x=pmr/a$. Then we have
$$
\alpha^2=\frac{a}{b}=\frac{a(1+x)}{b(1+x)}=\frac{a+pmr}{b+\frac{bpmr}{a}}
$$
Since $bmr \equiv r \pmod a$, by inequality~\eqref{line:ra}, we can find an integer $N$ such that $$pN+1\leq \frac{bpmr}{a} \leq pN+1.5.$$ Let $\alpha_{blue}=\sqrt{a+pmr}$ and $\alpha_{red}=\sqrt{b+\frac{bpmr}{a}}$; then $\alpha=\alpha_{blue}/\alpha_{red}$. Note that $\alpha_{red}^2 \in [1,1.5] \pmod p$, and $\alpha_{blue}^2$ is an integer not divisible by $p$ since $p \nmid a$.

In all of the above three cases, Corollary~\ref{cor:redblue} (1) implies that $\E^n \not \to (\ell_3, \alpha \ell_{p^2})$.
\end{proof}

Now we have all the ingredients to prove Theorem~\ref{thm:alpha}.

\begin{proof}[Proof of Theorem~\ref{thm:alpha}]
If $\alpha^2$ is irrational, then by Proposition~\ref{prop:rational} (1), we have $\E^n \not \to (\ell_3, \alpha \ell_{2209})$. Next, assume that  $\alpha^2=a/b$ is a rational number in simplest form. 

Note that if there is a prime $p \in \{47,59,67,71,73, 79,83\}$, such that $p \nmid b$ or $a>4p/3$, then we have $\E^n \not \to (\ell_3, \alpha \ell_{p^2})$ by Proposition~\ref{prop:rational} (2) and in particular $\E^n \not \to (\ell_3, \alpha \ell_{6889})$ holds. Next, we assume that this is not the case. Then we must have $a\leq \frac{4 \cdot 47}{3}$, that is, $a\leq 62$. Also, we have $p \mid b$ for each $p \in \{47,59,67,71,73, 79,83\}$. It follows that 
$$b \geq 47 \cdot 59 \cdot 67  \cdot 71 \cdot 73 \cdot 79 \cdot 83>4 \cdot 59^5 \cdot a.$$
Let $\alpha_{blue}=\sqrt{a+a/b}$ and $\alpha_{red}=\sqrt{b+1}$; then $\alpha=\alpha_{blue}/\alpha_{red}$. Note that $\alpha_{red}^2=b+1 \equiv 1 \pmod {59}$, and $\alpha_{blue}^2=a+\frac{a}{b}$ with $59 \nmid a$ and $0<a/b<(4 \cdot 59^5)^{-1}$. Now Corollary~\ref{cor:redblue} (2) implies that  $\E^n \not \to (\ell_3, \alpha \ell_{6845})$.  
\end{proof}

\section*{Acknowledgments}
The authors thank J\'ozsef Solymosi and Joshua Zahl for helpful discussions. The authors also thank the anonymous referees for their many valuable comments and suggestions, and in particular to Arsenii Sagdeev for helping us improve the results of theorem \ref{thm:l3}. 

\bibliographystyle{abbrv}
\bibliography{references}

\end{document}